\documentclass[11pt,amsart]{amsart}
\usepackage[all]{xy}
\usepackage{mathrsfs}
\usepackage{amsmath, amsthm, amssymb, amscd, amsgen,amsxtra,amsfonts, amsbsy}
\allowdisplaybreaks
\usepackage[all]{xy}
\usepackage{verbatim}

\usepackage{tikz} 
\usepackage{pgfplots}
\usepackage{tikz-cd}
\usepackage{etex}
\usepackage{calligra,mathrsfs}
\usepackage{color}
\definecolor{shadecolor}{gray}{0.875}
\definecolor{dblue}{rgb}{0,0, .6}
\usepackage[colorlinks=true, linkcolor=dblue, citecolor=dblue, filecolor = dblue, menucolor = dblue, urlcolor = dblue]{hyperref}
\usepackage{mathtools}
\usepackage{extarrows}
\usepackage{ifthen}
\usepackage{bm}
\usepackage{newtxtext,newtxmath}

\usepackage{graphicx}
\usepackage{caption,rotating}
\usepackage{color}
\usepackage{subcaption}
\usepackage{hyperref}
\usepackage{enumitem}
\usepackage[
  papersize={8.5in,11in},
  textwidth=30pc,
  textheight=50pc,
  centering
]{geometry}

\numberwithin{equation}{section}

\usepackage[toc,page]{appendix}

\renewcommand{\deg}{\operatorname{deg}}

\renewcommand{\min}{\operatorname{min}}
\renewcommand{\max}{\operatorname{max}}

\newcommand{\Q}{\mathbb{Q}}

\usepackage{url} 
\newtheorem{theorem}{Theorem}[section]

\newtheorem{corollary}[theorem]{Corollary}
\newtheorem{conjecture}[theorem]{Conjecture}
\newtheorem{proposition}[theorem]{Proposition}
\newtheorem*{theorem*}{Theorem}
\newtheorem*{corollary*}{Corollary}
\theoremstyle{definition}
\newtheorem{defn}[theorem]{Definition}

\newtheorem{remark}[theorem]{Remark}

\newtheorem{variant}[theorem]{Variant}

\newtheorem*{theoremA}{Theorem A}
\newtheorem*{theoremB}{Theorem B}
\newtheorem*{theoremC}{Theorem C}
\newtheorem*{theoremD}{Theorem D}
\newtheorem*{theoremE}{Theorem E}
\numberwithin{equation}{section}

\author{Houari Benammar Ammar}
\AtEndDocument{\bigskip{\footnotesize%
 \textsc{Department of Mathematical and Statistical Sciences, University of Alberta, Edmonton, Alberta T6G 2G1, Canada}
   \par  
  \textit{E-mail address :} \texttt{houari.math.benammar@gmail.com} 
}}
\title{Kodaira dimension of algebraic fiber spaces over threefolds : Part I}
\begin{document}

\normalsize
\maketitle
\begin{abstract}
We study the behavior of the Kodaira dimension of algebraic fiber spaces over threefolds. We prove some cases of the Iitaka Conjecture $C_{n,3}$, including certain situations where the base variety is a Calabi--Yau threefold.
\end{abstract}
\section{Introduction}
One of the remaining open problems in birational geometry is the Iitaka Conjecture, denoted by $C_{n,m}$.
\begin{conjecture}\label{Conj1}
Let $(X,\Delta)$ be a Klt pair. Let $f : X \to Y$ be an algebraic fiber space, where $X$ and $Y$ are smooth projective varieties of dimension $n$ and $m$,
respectively, and let $F$ be a general fiber of $f$. Then
\[
\kappa(K_X +\Delta) \geq \kappa(K_F + \Delta|_{F}) + \kappa(Y).
\]
\end{conjecture}
This conjecture allows us to understand the behavior of the Kodaira dimension in a family of varieties and to see the relationship between the Kodaira dimension of the total space, the general fiber, and the base variety. As shown in \cite{kawamataminimalmodel}, the Iitaka Conjecture follows from the existence of minimal models and the abundance conjecture. That is, if the general fiber has a good minimal model, then the Iitaka Conjecture holds. Therefore, it is much weaker in nature than the minimal model and abundance conjectures.

When the geometry of the base variety is well understood, one can derive certain cases of the Iitaka Conjecture even without much information about the general fiber. For instance, in \cite{Kawamatacurves}, Kawamata proved $C_{n,1}$, and Cao \cite{Caolitaka2} proved $C_{n,2}$ (initiated by Birkar \cite{Birkar_2009}). For higher-dimensional bases, there are not many results, except for some special varieties such as for general type varieties \cite{viehweglitaka}, and abelian varieties as shown in \cite{Caopaun}. This last case was later generalized to bases of large Albanese dimension in \cite{Houaricnm}. Motivated by these results, we initiated the project in the case of algebraic fiber spaces over threefolds.

In this paper, the proofs are presented in Sections~3, 4, 5, and~6. 
In Section~3, we prove Conjecture~\ref{Conj1} in the case where the base variety is an irregular threefold.
\begin{theoremA}[Corollary \ref{irregthm}]
Iitaka conjecture $C_{n,3}$ holds provided that $q(Y)>0$.
\end{theoremA}
In Section~4, we study Conjecture~\ref{Conj1} in the case where the general fiber satisfies
\[
\kappa(K_F + \Delta|_F) = 0.
\]
We use the canonical bundle formula (recalled in Theorem~\ref{fujinomoricbf}) and the reduction map (Theorem~\ref{reductionmap}) to derive the following theorem.
\begin{theoremB}[Theorem  \ref{fiberkod0}]
Let $(X, \Delta)$ be a klt pair, and let 
$f \colon X \to Y$
 be an algebraic fiber space, where $X$ and $Y$ are smooth projective varieties of dimensions $n$ and $3$, respectively. Let $F$ be a general fiber of $f$ such that 
\[
\kappa(K_F + \Delta|_F) = 0.
\] 
In the setting of Theorem~\ref{fujinomoricbf}, one has
\[
\kappa(K_X + \Delta) \ge \kappa(Y),
\] 
except maybe the following two cases, when $L$ is not numerically $\Q$-effective and
\begin{enumerate}
    \item $\nu(L) = 1$, $\eta(L) = 2$, and $q(Y) = 0$; or
    \item $L$ is almost strictly nef and $\widetilde{Y}$ is a Calabi--Yau threefold with $L \cdot c_2(\widetilde{Y}) = 0$.
\end{enumerate}
\end{theoremB}
Here, $\eta(\cdot)$ and $\nu(\cdot)$ denote the nef dimension and the numerical dimension, respectively, recalled in Section~2. 
The divisor $L$ is the moduli part of the fibration $f$. It is still unknown whether $L$ is semi-ample, except in certain situations, for instance in the case of parabolic fibrations \cite{Benbakker}. 
At present, we only know that $L$ is nef. 

When $L$ is numerically $\Q$-effective, we apply \cite[Theorem~0.1]{campanapaunkoziarz} to deduce the theorem. 
In cases~$(1)$ and~$(2)$, we cannot determine whether $L$ is numerically $\Q$-effective or not. As an application of Theorem B, we deduce almost $C_{7,m}$, and this is done in Corollary \ref{iitakac_7,m} (see also \cite{chikangchang}). 

In Section~5, we allow
\[
\kappa(K_F + \Delta|_F) \geq 0,
\]
while working over threefolds with $\kappa(Y) > 0$. With the help of Proposition~\ref{caoproposition}, we prove the following theorems.
\begin{theoremC}[Theorem \ref{positivekodaira}]
Let $(X,\Delta)$ be a Klt pair. Let $f : X \to Y$ be an algebraic fiber space, where $X$ and $Y$ are smooth projective varieties of dimension $n$ and $3$,
respectively such that $\kappa(Y)\geq 2$, and let $F$ be a general fiber of $f$. Then
\[
\kappa(K_X +\Delta) \geq \kappa(K_F + \Delta|_{F}) + \kappa(Y).\]
\end{theoremC}
\begin{theoremD}[Theorem \ref{kodairadimension1}]
Let $(X,\Delta)$ be a klt pair. Let $f \colon X \to Y$ be an algebraic fiber space, where $X$ and $Y$ are smooth projective varieties of dimensions $n$ and $3$, respectively, such that $\kappa(Y)=1$, and let $F$ be a general fiber of $f$. Then
\[
\kappa(K_X + \Delta) \geq \kappa(K_F + \Delta|_{F}) + 1,
\]
provided that
\[
\kappa\!\left(\det f_{*}\mathcal{O}_X\!\left(m(K_{X/Y} + \Delta)\right)\big|_{Y_Z}\right) > 0, \hspace{0.2cm} \text{for $m$ sufficiently divisible},
\]
in the case where $Z \simeq \mathbb{P}^1$.    
\end{theoremD}
Here, $Y_Z$ denotes a general fiber of the Iitaka fibration 
$g \colon Y \to Z$ of $Y$, we refer to Section~5 for more details. 
By Theorems~B, C, and~D, we can improve Theorem~B by adding the condition 
$\kappa(Y) \leq 1$ in Case~1 of Theorem~B.

We encounter some difficulties in dropping the condition in Theorem~D. 
However, we can still deduce the Iitaka inequality in some situations, where 
$\kappa(Y)=1$, $Z \simeq \mathbb{P}^1$, and
\[
\kappa\!\left(\det f_{*}\mathcal{O}_X\!\left(m(K_{X/Y} + \Delta)\right)\big|_{Y_Z}\right) \leq 0,
\]
we refer to Remark~\ref{remarksection5} for details.

In the last section, we study the case when the threefold variety $Y$ satisfies $\kappa(Y) = q(Y) = 0$. Up to a finite étale cover, it remains to prove the Iitaka inequality when $Y$ is a Calabi–Yau threefold. First, we are able to prove the following theorem.

\begin{theoremE}[Theorem \ref{overCY}]
Let $(X,\Delta)$ be a Klt pair. Let $f : X \to Y$ be an algebraic fiber space, where $X$ is a smooth projective variety of dimension $n$ and 
$Y$ is a Calabi-Yau threefold, and let $F$ be a general fiber of $f$. Then
\[
\kappa(K_X +\Delta) \geq \kappa(K_F + \Delta|_{F}),\]
provided that $\kappa(\det f_{*}\mathcal{O}_X\bigl(m(K_{X/Y} + \Delta)\bigr) > 1$ for $m$ sufficiently divisible.    
\end{theoremE}
In the proof of this theorem, we use the bundle 
$\det f_{*}\mathcal{O}_X\bigl(m(K_{X/Y}+\Delta)\bigr)$ 
to fiber (after resolution) the Calabi–Yau variety $Y$ in the case when
\[
\kappa\Bigl(\det f_{*}\mathcal{O}_X\bigl(m(K_{X/Y} + \Delta)\bigr)\Bigr) = 2.
\]

We partially prove the Iitaka inequality
\[
\kappa(K_X + \Delta) \geq \kappa(K_F + \Delta|_{F}),
\]
in the case when $Y$ is a Calabi–Yau threefold and
\[
0 \leq \kappa\Bigl(\det f_{*}\mathcal{O}_X\bigl(m(K_{X/Y} + \Delta)\bigr)\Bigr) \leq 1,
\]
as established in Proposition~\ref{propoovercy} and Remarks~\ref{remark1overcy} and~\ref{remarkovercy2}. It is expected that $
\det f_{*}\mathcal{O}_X\bigl(m(K_{X/Y} + \Delta)\bigr)$
is $\mathbb{Q}$-effective when the base variety is a Calabi--Yau threefold. However, it is not yet completely known whether this holds for every Calabi--Yau threefold. We refer to the end of Section~6 for more details. This question will be discussed in the author's forthcoming work.

In the preliminary Section~2, we recall some results, notation, and definitions needed for the proofs of the above theorems.

\subsection*{Acknowledgments}
I would like  to thank Jungkai Chen for his kind invitation in March 2026 and to NCTS for its  hospitality. I have  benefited from  discussions on this subject with Chi-Kang Chang, Jungkai Chen, Xi Chen, Ching-Jui Lai, Hsueh-Yung Lin, and Steven Lu. 

\section{Background}
\subsection{Positivity of Pushforward Sheaves and Extension Theorems}
One powerful result in algebraic geometry, which is helpful for the Iitaka program, is a good extension theorem. In \cite{Caolitaka2}, the author used a consequence of the Ohsawa--Takegoshi extension theorem to prove \(C_{n,2}\), which is recalled in the proposition below. We first recall a deeper version of Fujita-type results on the positivity of direct image sheaves. This builds on the work of many authors.
\begin{theorem}[\cite{HSP},{\cite{pauntakayama}}]\label{positivity results}
Let $f \colon X \to Y$ be an algebraic fiber space between smooth varieties, with general fiber $F$. Let $D$ be a pseudo-effective divisor on $X$ endowed with a singular Hermitian metric $h_D$ of semi-positive curvature, i.e.,
\[
i\Theta_{h_D}(\mathcal{O}_X(D)) \geq 0.
\]
Then the direct image sheaf
\[
f_{*}\bigl(\omega_{X/Y} \otimes \mathcal{O}_X(D) \otimes \mathcal{J}(h_D)\bigr)
\]
admits a singular Hermitian metric $h$ such that
\[
i\Theta_{h}\bigl(f_{*}(\omega_{X/Y} \otimes \mathcal{O}_X(D) \otimes \mathcal{J}(h_D))\bigr) \geq 0.
\]
Moreover, the determinant line bundle
\[
\det f_{*}\bigl(\omega_{X/Y} \otimes \mathcal{O}_X(D) \otimes \mathcal{J}(h_D)\bigr)
\]
is pseudo-effective, and its induced metric satisfies
\[
i\Theta_{\det h}\bigl(\det f_{*}(\omega_{X/Y} \otimes \mathcal{O}_X(D) \otimes \mathcal{J}(h_D))\bigr) \geq 0.
\]

\end{theorem}
\begin{variant}\label{variant1}
If $\kappa(F) \geq 0$, then $K_{X/Y}$ is pseudo-effective. It is well known that $K_{X/Y}$ admits a canonical singular Hermitian metric with semi-positive curvature, called the Narasimhan--Simha metric. If $h_D$ denotes the induced metric on
\[
D = (m-1)K_{X/Y}, \quad m \geq 2,
\]
then there is a natural inclusion
\[
f_{*}\bigl(\omega_{X/Y} \otimes \mathcal{O}_X(D) \otimes \mathcal{J}(h_D)\bigr)
\hookrightarrow
f_{*}\omega_{X/Y}^{\otimes m},
\]
which is generically an isomorphism over the smooth locus of $f$. In particular, the sheaf
$f_{*}\omega_{X/Y}^{\otimes m}$ admits a singular Hermitian metric, denoted by $h$ (without loss of generality), such that
\[
i\Theta_{h}\bigl(f_{*}\omega_{X/Y}^{\otimes m}\bigr) \geq 0
\]
and
\[
i\Theta_{\det h}\bigl(\det f_{*}\omega_{X/Y}^{\otimes m}\bigr) \geq 0.
\]

\end{variant}
\begin{variant}\label{variant 2}
From Theorem~\ref{positivity results} and Variant~\ref{variant1}, we see that for any pseudo-effective divisor $D$ endowed with a singular Hermitian metric $h_D$, the direct image sheaf
\[
f_{*}\bigl(\omega_{X/Y}^{\otimes m} \otimes \mathcal{O}_X(D)\bigr)
\]
admits a singular Hermitian metric $h$ such that
\[
i\Theta_{h}\bigl(f_{*}(\omega_{X/Y}^{\otimes m} \otimes \mathcal{O}_X(D))\bigr) \geq 0
\]
for all sufficiently large $m \in \mathbb{N}$. Indeed, for $m$ large enough and for a general fiber $F$, we have
\[
\mathcal{J}\bigl(h_D^{1/m}\big|_{F}\bigr) = \mathcal{O}_F.
\]
\end{variant}
\begin{variant}\label{variant3}
Let $(X,\Delta)$ be a klt pair such that $K_X + \Delta$ is $\mathbb{Q}$-Cartier, and let $D$ be a pseudo-effective divisor on $X$. Then, for all sufficiently large $m \in \mathbb{N}$, the direct image sheaf
\[
f_{*}\mathcal{O}_X\bigl(m(K_{X/Y} + \Delta) + D\bigr)
\]
admits a singular Hermitian metric $h$ such that
\[
i\Theta_{h}\bigl(f_{*}\mathcal{O}_X(m(K_{X/Y} + \Delta) + D)\bigr) \geq 0,
\]
and
\[
i\Theta_{\det h}\bigl(\det f_{*}\mathcal{O}_X(m(K_{X/Y} + \Delta) + D)\bigr) \geq 0.
\]
\end{variant}
We now recall the following proposition from \cite{Caolitaka2}.
\begin{proposition}[{\cite[Proposition 2.9]{Caolitaka2}}]\label{caoproposition}
In the setting of Theorem~\ref{positivity results} and Variants~\ref{variant1}, \ref{variant 2}, and \ref{variant3}, we assume moreover that there exists an algebraic fiber space
\[
g \colon Y \to Z
\]
onto a projective variety \( Z \). Let \( H \) be a pseudo-effective divisor on \( Y \) endowed with a singular metric \( h_H \) such that
\[
i\Theta_{h_H}(\mathcal{O}_{Y}(H)) \geq 0 .
\]
Let \( A_Z \) be an ample line bundle on \( Z \). Then, for \( c \in \mathbb{N} \) sufficiently large (depending only on \( A_Z \) and \( Z \)), the following holds: for every
\[
s \in H^{0}\!\left(Y_Z,\, \omega_{Y} \otimes \mathcal{O}_{Y}(H) \otimes f_{*}\mathcal{O}_{X}\bigl(m(K_{X/Y} + \Delta) + D\bigr)\right)
\]
such that
\[
\int_{Y_Z} |s|^{2}_{h_H, h} < \infty,
\]
there exists
\[
S \in H^{0}\!\left(Y,\, \omega_{Y} \otimes \mathcal{O}_{Y}(H) \otimes f_{*}\mathcal{O}_{X}\bigl(m(K_{X/Y} + \Delta) + D\bigr) \otimes g^{*}(cA_Z)\right)
\]
such that
\[
S|_{Y_Z} = s \otimes g^{*}e,
\]
where \( e \in \mathcal{O}_{Z,z}(cA_Z) \).

\end{proposition}
\begin{remark}
Proposition~\ref{caoproposition} can be applied to prove Iitaka Conjecture~\ref{Conj1} when the base variety~$Y$  admits a good minimal model. In particular,  for threefolds base by the works of \cite{kawamataabundance}, \cite{Miyaoka}, and \cite{moriminimalthreefold}.
\end{remark}
\begin{proposition}[{\cite[Proposition 2.11]{Caolitaka2}}]\label{caoprop2}
Let \( E \) be a holomorphic vector bundle on a projective variety \( X \), and let \( h_E \) be a singular Hermitian metric on \( E \) with semi-positive curvature. Let \( U \subset X \) be a (topological) open subset. If
\[
i\Theta_{\det h_E}\bigl(\det E\bigr) \equiv 0 \quad \text{on } U,
\]
then \( h_E \) is smooth on \( E|_{U} \), and \( (E|_{U}, h_E) \) is Hermitian flat.
\end{proposition}
In suitable situations, we can apply the following theorem due to \cite{schnellcampanaconj}.
\begin{theorem}[{\cite[Theorem 12.1]{schnellcampanaconj}}]\label{schnellthm}
Let $f: X \to Y$ be an algebraic fiber space with $\kappa(F) \geq 0$. Let $H$
be an ample divisor on $Y$. If $mK_{X} - f^{*}H$ is pseudo-eﬀective for some $m\geq 1$, then
$m K_{X} - f^{*}H$ becomes eﬀective for $m$ suﬃciently large and divisible, provided that $Y$ is not uniruled.
\end{theorem}
\begin{remark}\label{kltremarks}
Clearly, the proof of Theorem~\ref{schnellthm} in \cite{schnellcampanaconj} works in the klt setting.
\end{remark}
Now we state the Catanese--Fujita--Kawamata decomposition theorem for pushforward sheaves, proved in \cite{catanese2}, \cite{catanese3}, \cite{fujitapaper}, and \cite{Schnell-lombardi}. 
\begin{theorem}[{\cite[Theorem 1.2]{Schnell-lombardi}}]\label{fujitadecom}
     Let \( f: X \to Y \) be a surjective morphism, and let \( (X, \Delta) \) be a klt pair. Then the torsion free sheaf \( f_{*}\mathcal{O}_{X}(m(K_{X/Y} + \Delta)) \)  admits the Catanese-Fujita-Kawamata decomposition
    \[
    f_{*}\mathcal{O}_{X}(m(K_{X/Y} + \Delta)) = \mathcal{A} \oplus \mathcal{U}\text{.}
    \]
    where $\mathcal{A}$ is a generically ample sheaf and $\mathcal{U}$ is a Hermitian flat bundle.
\end{theorem}
\par \subsection{Fujino-Mori canonical bundle formula} During the proofs, we will apply the Fujino--Mori canonical bundle formula.
\begin{theorem}[{\cite{FujinoGongyo}}, {\cite[Theorem~4.5]{fujinomori}}]\label{fujinomoricbf}
Let $f \colon X \to Y$ be an algebraic fiber space. Let $(X,\Delta)$ be a klt pair,  and let $F$ be a general fiber of $f$ such that
\[
\kappa\bigl(K_F + \Delta|_{F}\bigr) = 0.
\]
Then there exists a commutative diagram
\[
\begin{tikzcd}
\widetilde{X} \arrow[r, "\widetilde{f}"] \arrow[d, "\sigma"] &
\widetilde{Y} \arrow[d, "\tau"] \\
X \arrow[r, "f"] & Y
\end{tikzcd}
\]
with the following properties:
\begin{itemize}
    \item[(1)] $\widetilde{f}$ is an algebraic fiber space, and $\sigma$ and $\tau$ are birational morphisms.
    
    \item[(2)] There exists a $\mathbb{Q}$-divisor $\widetilde{\Delta}$ such that $(\widetilde{X}, \widetilde{\Delta})$ is a klt pair and
    \[
    \sigma_{*}\bigl(p(K_{\widetilde{X}} + \widetilde{\Delta})\bigr)
    =
    p(K_X + \Delta)
    \]
    for a sufficiently large integer $p \in \mathbb{N}$.
    
    \item[(3)] There exist a $\mathbb{Q}$-effective divisor $B$ and a nef $\mathbb{Q}$-divisor $L$ on $\widetilde{Y}$ such that $(\widetilde{Y}, B+L)$ is a klt pair. Moreover, there exists a $\mathbb{Q}$-divisor $R = R^{+} - R^{-}$ on $\widetilde{X}$, decomposed into its positive and negative parts, such that
    \[
    p(K_{\widetilde{X}} + \widetilde{\Delta})
    =
    p\,\widetilde{f}^{*}(K_{\widetilde{Y}} + B + L) + R,
    \]
    for a sufficiently large integer $p \in \mathbb{N}$.
    
    \item[(4)] $\widetilde{f}_{*}\mathcal{O}_{\widetilde{X}}(iR^{+})
    =
    \mathcal{O}_{\widetilde{Y}}$ for any $i \in \mathbb{N}$.
    
    \item[(5)] The divisor $R^{-}$ is exceptional over $X$, and the codimension of
    \[
    \widetilde{f}\bigl(\operatorname{Supp} R^{-}\bigr)
    \]
    in $\widetilde{Y}$ is at least $2$.
\end{itemize}
\end{theorem}
\subsection{Reduction map}The (numerical) effectivity of the moduli part in the canonical bundle formula plays an important role in verifying the Iitaka conjecture, since it implies in some way the effectivity of the direct image sheaf of the relative pluricanonical bundle. However, this numerical effectivity is not clear in certain situations. Instead, it is important to understand to what extent the moduli part is nef. For this reason, we recall the notions of numerical dimension and nef dimension of divisors. By the work of many authors, one can construct a map, called the pseudo-effective (nef) reduction map, which measures the directions along which the divisor is numerically trivial.

\begin{defn}\label{nefdimension}
Let $D$ be a pseudo-effective $\mathbb{Q}$-Cartier divisor on a smooth projective variety $X$ of dimension $n$. The numerical dimension $\nu(D)$ is defined by
\[
\nu(D)
:=
\max \left\{
k \in \mathbb{N}
\;\middle|\;
\limsup_{m \to +\infty}
\frac{h^{0}\!\left(X, \mathcal{O}_{X}(\lfloor mD \rfloor + A)\right)}{m^{k}}
> 0
\right\},
\]
for any sufficiently ample divisor $A$.
If $D$ is nef, then the above quantity is equivalent to
\[
\nu(D)
:=
\max \left\{
k \in \mathbb{N}
\;\middle|\;
D^{k} \cdot A^{n-k} \neq 0
\right\}
\]
for some (hence any) ample divisor $A$.
\end{defn}
Recall that $\kappa(D) \leq \nu(D) \leq \dim X$, and that $D$ is big if and only if $\nu(D) = \dim X$. 
In order to recall the definition of the reduction map in the general setting, we need to use the divisorial Zariski decomposition. For a pseudo-effective divisor $D$ on a surface, there exists a classical decomposition of $D$ into a positive part and a negative part satisfying a precise orthogonality property, this is due to Zariski and Fujita. In higher dimensions, this theory was generalized by Nakayama \cite{nakayamabook} and Boucksom \cite{boucksom}.
\begin{defn}
Let $X$ be a smooth projective variety, and let $D$ be a pseudo-effective $\mathbb{Q}$-Cartier divisor on $X$. Fix an ample divisor $A$ on $X$.  
For a prime divisor $\Gamma \subset X$, define
\[
\sigma_\Gamma(D)
\coloneqq
\min \left\{
\operatorname{mult}_\Gamma(D')
\;\middle|\;
D' \ge 0,\;
D' \sim_{\mathbb{Q}} D + \varepsilon A \text{ for some } \varepsilon > 0
\right\}.
\]
This definition is independent of the choice of $A$.
\end{defn}
There are only finitely many prime divisors $\Gamma$ such that $\sigma_\Gamma(D) > 0$. Hence we have the following definition.
\begin{defn}[{\cite[Definition 2.13]{Briannlehmann}}]
Let $X$ be a smooth projective variety, and let $D$ be a pseudo-effective $\mathbb{Q}$-Cartier divisor on $X$. Define the $\mathbb{Q}$-Cartier divisors
\[
N_\sigma(D) \coloneqq \sum_{\Gamma} \sigma_\Gamma(D)\,\Gamma
\quad\text{and}\quad
P_\sigma(D) \coloneqq D - N_\sigma(D),
\]
where the sum is taken over all prime divisors $\Gamma \subset X$.

The decomposition
\[
D = P_\sigma(D) + N_\sigma(D)
\]
is called the \emph{divisorial Zariski decomposition} of $D$.
\end{defn}
Recall that $\kappa(D) = \kappa\bigl(P_\sigma(D)\bigr)$ and $\nu(D) = \nu\bigl(P_\sigma(D)\bigr)$. Moreover, the support of the negative part $N_\sigma(D)$ coincides with the divisorial component of the diminished base locus of $D$.

\begin{theorem}[{\cite[Theorem 1.3]{Briannlehmann}}]\label{reductionmap}
Let $Y$ be a smooth projective variety, and let $D$ be a pseudo-effective $\mathbb{Q}$-Cartier divisor on $Y$. 
There exist a birational morphism $\varphi \colon \bar{Y} \to Y$ and a surjective morphism
\[
\pi \colon \bar{Y} \to Z
\]
with connected fibers satisfying the following properties:
\begin{enumerate}
\item For a general fiber $G$ of $\pi$, one has
\[
P_\sigma(\varphi^{*}D)\big|_{G} \equiv 0.
\]

\item The pair $(\bar{Y}, \pi)$ is the maximal quotient satisfying {\rm(1)} in the following sense:
if $\varphi' \colon Y' \to Y$ is a birational morphism and
$\pi' \colon Y' \to Z'$ is a surjective morphism with connected fibers
satisfying {\rm(1)}, then there exists a dominant rational map
\[
\psi \colon Z' \dashrightarrow Z
\]
such that $\pi$ is birationally equivalent to $\psi \circ \pi'$.
\end{enumerate}

The pair $(\bar{Y}, \pi)$ is determined up to birational equivalence and depends only on
the numerical class of $D$. The induced rational map
\[
\pi \circ \varphi^{-1} \colon Y \dashrightarrow Z
\]
is called the \emph{pseudo-effective reduction map} associated to $D$.
\end{theorem}
If $D$ is nef, then the nef dimension $\eta(D)$ of $D$ is equal to $\dim Z$ by definition. 
The pseudo-effective reduction map above was first defined rigorously in \cite{Briannlehmann}. 
For important reduction steps, it is useful to understand in which directions the divisor is numerically trivial, and whether it is numerically equivalent to a divisor pulled back from the base. 
Hence, we recall the following theorem.
\begin{theorem}[{\cite[Theorem~5.3]{lehmannreduction2}}]\label{fibertrivial}
Let $f \colon Y \to Z$ be a surjective morphism of smooth projective varieties with connected fibers. 
Suppose that $D$ is a pseudo-effective $\mathbb{Q}$-Cartier divisor such that
\[
P_\sigma(D)\big|_F \equiv 0
\]
for a general fiber $F$ of $f$.

Then there exist a smooth birational model
$\varphi \colon Y' \to Y$, a morphism
$g \colon Y' \to Z'$ birationally equivalent to $f$, 
and a $\mathbb{Q}$-Cartier divisor $D'$ on $Z'$ such that
\[
P_\sigma(\varphi^* D) \equiv P_\sigma(g^* D').
\]
\end{theorem}

\section{Algebraic fiber spaces over irregular Threefolds}
In this Section, we give a proof of Conjecture \ref{Conj1} in case of $q(Y)>0$. We recall the definition of Albanese dimension.
\begin{defn}
     We define the Albanese dimension $\alpha(X)$ of an irregular variety $X$ by 
    \[
    \alpha(X):= \dim \operatorname{alb}_{X}(X).
    \]
    Here, $\operatorname{alb}_{X}(X)$ is the image of the Albanese map $\operatorname{alb}_{X}: X \to \operatorname{alb}_{X}(X) \subseteq \operatorname{Alb}(X)$, where $\operatorname{Alb}(X)$ is the Albanese variety.
\end{defn}
In \cite{Houaricnm}, we proved the following Theorem. 
\begin{theorem}[{\cite[Theorem 1.2]{Houaricnm}}]\label{maintheorem}
Iitaka conjecture $C_{n,m}$ holds when $\alpha(Y) \geq m-2$.
\end{theorem}
\begin{remark}
In \cite{Houaricnm}, we considered a smooth variety $X$, instead of a Klt pair $(X, \Delta)$. However, the same proof given in \cite{Houaricnm} obviously works for the Klt pair $(X, \Delta)$.  
\end{remark}
Therefore, we can deduce the following Corollary.
\begin{corollary}\label{irregthm}
Iitaka conjecture $C_{n,3}$ holds provided that $q(Y)>0$.
\end{corollary}
\begin{proof}
If $\alpha(Y)=3$, then $\operatorname{alb}_{Y}$ is generically finite, and the result from \cite{Caopaun} and \cite{HSP}. If $\alpha(Y)=1,2$, we are in the situation of Theorem \ref{maintheorem}. Thus the theorem is proved.
\end{proof}
\section{Algebraic Fiber Spaces over Threefolds with General Fiber of Log Kodaira Dimension Zero}

In this section, we apply the canonical bundle formula to derive some results on Iitaka conjecture $C_{n,3}$, in the case where the general fiber has log Kodaira dimension zero. In particular, as a corollary, we almost deduce $C_{7,m}$.
\begin{theorem}\label{fiberkod0}
Let $(X, \Delta)$ be a klt pair, and let 
$f \colon X \to Y$
 be an algebraic fiber space, where $X$ and $Y$ are smooth projective varieties of dimensions $n$ and $3$, respectively. Let $F$ be a general fiber of $f$ such that 
\[
\kappa(K_F + \Delta|_F) = 0.
\] 
In the setting of Theorem~\ref{fujinomoricbf}, one has
\[
\kappa(K_X + \Delta) \ge \kappa(Y),
\] 
except maybe the following two cases, when $L$ is not numerically $\Q$-effective and
\begin{enumerate}
    \item $\nu(L) = 1$, $\eta(L) = 2$, and $q(Y) = 0$; or
    \item $L$ is almost strictly nef and $\widetilde{Y}$ is a Calabi--Yau threefold with $L \cdot c_2(\widetilde{Y}) = 0$.
\end{enumerate}
\end{theorem}
\begin{proof}
By Corollary~\ref{irregthm}, we may assume that $q(Y) = 0$. We argue depending on the nef dimension $\eta(L)$.

\medskip
\noindent
\textbf{Case 1:} $\eta(L) = 0$ or $1$.  

If $\eta(L) = 0$, then $L \equiv 0$. If $\eta(L) = 1$, then by \cite[Paragraph~2.4.4]{reductionmap}, the reduction map 
\[
g \colon \widetilde{Y} \to Z
\] 
is regular, where $Z$ is a smooth projective curve, and there exists a $\mathbb{Q}$-effective divisor $D$ on $Z$ with $\deg D > 0$ such that 
\[
L \equiv g^* D.
\] 
In either situation, $L \equiv D'$ for some $\mathbb{Q}$-effective divisor $D'$. We then set
\[
M := \sigma_* \widetilde{f}^*(D' - L),
\] 
and for sufficiently divisible $p$ and $i$, we have
\[
H^0\bigl(ip(K_X + \Delta + M)\bigr)
= H^0\bigl(ip(K_{\widetilde{X}} + \widetilde{\Delta} + \widetilde{f}^*(D' - L) + i R^-)\bigr)
= H^0\bigl(ip(K_{\widetilde{Y}} + B + D')\bigr).
\] 
Applying \cite[Theorem~0.1]{campanapaunkoziarz}, we deduce
\[
\kappa(K_X + \Delta) \ge \kappa(K_X + \Delta + M) = \kappa(K_{\widetilde{Y}} + B + D') \ge \kappa(Y),
\] 
as desired.

\medskip
\noindent
\textbf{Case 2:} $\eta(L) = 2$.  

There are two possibilities: $\nu(L) = 1$ or $\nu(L) = 2$. By assumption, $\nu(L) = 2$. Applying Theorem~\ref{reductionmap}, there exists a birational model 
\[
\varphi \colon \bar{Y} \to \widetilde{Y}
\] 
and a surjective morphism
\[
\pi \colon \bar{Y} \to Z
\] 
with connected fibers such that, for a general fiber $G$ of $\pi$,
\[
P_\sigma(\varphi^* L)|_G \equiv 0.
\] 
Then, by Theorem~\ref{fibertrivial}, there exist a smooth birational model
\[
\psi \colon Y' \to \bar{Y},
\]
a morphism
\[
g \colon Y' \to Z'
\] 
birationally equivalent to $\pi$, and a $\mathbb{Q}$-Cartier divisor $D$ on $Z'$ such that
\begin{equation}\label{eq1thm4.1}
P_\sigma((\varphi \circ \psi)^* L) \equiv P_\sigma(g^* D).
\end{equation}
We have $\dim Z' = 2$ and $\nu(D) = 2$, hence $D$ is big on $Z'$ and $P_\sigma(g^* D)$ is $\Q$-effective. Therefore, $(\varphi \circ \psi)^* L$ is numerically equivalent to a $\mathbb{Q}$-effective divisor
\[
D' := N_\sigma((\varphi \circ \psi)^* L) + P_\sigma(g^* D).
\]

\medskip
\noindent
Consider the diagram:
\[
\begin{tikzcd}
X' \arrow[r, "f'"] \arrow[d, "\bar{\sigma}"] & Y' \arrow[r, "g"] \arrow[d, "\psi"] & Z' \\
\bar{X} \arrow[r, "\bar{f}"] \arrow[d, "\widetilde{\sigma}"] & \bar{Y} \arrow[r, "\pi"] \arrow[d, "\varphi"] & Z \\
\widetilde{X} \arrow[r, "\widetilde{f}"] \arrow[d, "\sigma"] & \widetilde{Y} \arrow[d, "\tau"] \\
X \arrow[r, "f"] & Y
\end{tikzcd}
\]
The varieties $\bar{X}$ and $X'$ are birational to $X$, and there exists a klt pair $(X', \Delta')$ such that
\[
K_{X'} + \Delta' = (\widetilde{\sigma} \circ \bar{\sigma})^*(K_{\widetilde{X}} + \widetilde{\Delta}).
\] 
Hence,
\[
p(K_{X'} + \Delta') - (\widetilde{\sigma} \circ \bar{\sigma})^* R
= p f'^* (\varphi \circ \psi)^* (K_{\widetilde{Y}} + B + L)
\] 
We then set
\[
M := (\sigma \circ \widetilde{\sigma} \circ \bar{\sigma})_* f'^* (P_\sigma(g^* D) - P_\sigma((\varphi \circ \psi)^* L)).
\] 
By \eqref{eq1thm4.1} and for sufficiently divisible $p$ and $i$, we obtain
\[
H^0\bigl(ip(K_X + \Delta + M)\bigr) = H^0\bigl(ip(K_{X'} + \Delta' + \widetilde{f}^*(P_\sigma(g^* D) - P_\sigma((\varphi \circ \psi)^* L)) + (\widetilde{\sigma} \circ \bar{\sigma})^{*} iR^-)\bigr)\]
\[\geq H^0\bigl(ip(K_{\widetilde{Y}} + B)\bigr). 
\] 
Applying \cite[Theorem~0.1]{campanapaunkoziarz}, we deduce
\[
\kappa(K_X + \Delta) \ge \kappa(K_X + \Delta + M) \ge \kappa(Y).
\]

\medskip
\noindent
\textbf{Case 3:} $\eta(L) = 3$.  
In this case $L$ is almost stricly nef. By \cite{serranoconj}, $K_{\widetilde{Y}} + L$ is big, except when $\widetilde{Y}$ is a Calabi--Yau threefold with 
\[
L \cdot c_2(\widetilde{Y}) = 0.
\] 
Hence, by Theorem~\ref{fujinomoricbf}, we deduce
\[
\kappa(K_X + \Delta) \ge 3,
\]
completing the proof.
\end{proof}
\begin{corollary}\label{iitakac_7,m}
Let $(X, \Delta)$ be a klt pair, and let 
\[
f \colon X \to Y
\] 
be an algebraic fiber space, where $X$ and $Y$ are smooth projective varieties of dimensions $7$ and $m$, respectively. Let $F$ denote a general fiber of $f$. Then 
\[
\kappa(K_X + \Delta) \ge \kappa(K_F + \Delta|_F) + \kappa(Y),
\]
except maybe the case $m = 3$ and under the setting of Theorem~\ref{fujinomoricbf} when $L$ is not numerically $\Q$-effective:
\begin{enumerate}
    \item $\nu(L) = 1$, $\eta(L) = 2$, and $q(Y) = 0$;
    \item $L$ is almost strictly nef and $\widetilde{Y}$ is a Calabi--Yau threefold with $L \cdot c_2(\widetilde{Y}) = 0$.
\end{enumerate}
\end{corollary}
\begin{proof}
If $m = 1$ or $2$, the statement follows from Kawamata \cite{Kawamatacurves} and Cao \cite{Caolitaka2}, respectively.  
If $m \ge 4$, then the log general fiber $(F, \Delta|_F)$ admits a good minimal model, and applying \cite{kawamataminimalmodel}, the result holds.  

The only remaining case is $m = 3$, so $\dim F = 4$.  
If $\kappa(K_F + \Delta|_F) > 0$, then $(F, \Delta|_F)$ has a good minimal model, and by \cite{kawamataminimalmodel}, the result again holds.  

Thus, we need only consider the case 
\[
\kappa(K_F + \Delta|_F) = 0,
\] 
and in this situation, the statement follows directly from Theorem~\ref{fiberkod0}.
\end{proof}
\section{Algebraic Fiber Spaces over Threefolds of Positive Kodaira Dimension}
In this section, we consider the case where the log general fiber $(F, \Delta|_F)$ has non-negative Kodaira dimension, that is,
\[
\kappa(K_F + \Delta|_F) \geq 0,
\]
and assuming that $\kappa(Y) > 0$. We will take the advantage of the existence of a good minimal model for the base variety $Y$. 
Under these assumptions, we prove the following theorems.
\begin{theorem}\label{positivekodaira}
Let $(X,\Delta)$ be a Klt pair. Let $f : X \to Y$ be an algebraic fiber space, where $X$ and $Y$ are smooth projective varieties of dimension $n$ and $3$,
respectively such that $\kappa(Y)\geq 2$, and let $F$ be a general fiber of $f$. Then
\[
\kappa(K_X +\Delta) \geq \kappa(K_F + \Delta|_{F}) + \kappa(Y).\]
\end{theorem}
\begin{proof}
We may assume that \( K_Y \) is semi-ample. If \( \kappa(Y) = 3 \), then the result follows from \cite{viehweglitaka}. If \( \kappa(Y) = 2 \), we consider the Iitaka fibration \( g \) of \( Y \) defined by \( |rK_Y| \) for \( r \) sufficiently large, and obtain the following diagram:
\[
\begin{tikzcd}
X \arrow[r, "f"] \arrow[rd, "h"] & Y \arrow[d, "g"] \\
 & Z
\end{tikzcd}
\]
with \( \dim Z = 2 \), and
\[
rK_Y = g^{*}A
\]
for some ample divisor \( A \) on \( Z \). We denote by \( Y_Z \) a general fiber of \( g \), it is an elliptic curve.  

If \( Z \) is not uniruled, then the desired result follows readily from Theorem~\ref{schnellthm}. Indeed, since
\[
\kappa\bigl(K_F + \Delta|_{F}\bigr) \geq 0,
\]
the divisor \( K_{X/Y} + \Delta \) is pseudo-effective, and we can write
\[
r(K_{X/Y} + \Delta) = r(K_X + \Delta) - h^{*}A.
\]
Therefore, by applying Theorem~\ref{schnellthm}, we deduce that \( r(K_X + \Delta) - h^{*}A \) is effective for \( r \) sufficiently large. Hence
\begin{equation}\label{eq3}
\kappa(K_X + \Delta)
= \kappa\bigl(K_{X_Z} + \Delta|_{X_Z}\bigr) + \dim Z
= \kappa\bigl(K_{X_Z} + \Delta|_{X_Z}\bigr) + \kappa(Y).
\end{equation}
Applying \cite{Kawamatacurves}, we obtain
\begin{equation}\label{eq4}
\kappa\bigl(K_{X_Z} + \Delta|_{X_Z}\bigr)
\geq \kappa\bigl(K_F + \Delta|_{F}\bigr) + \kappa(Y_Z)
= \kappa\bigl(K_F + \Delta|_{F}\bigr).
\end{equation}
By inequalities~\eqref{eq3} and~\eqref{eq4}, we conclude that
\[
\kappa(K_X + \Delta)
\geq \kappa\bigl(K_F + \Delta|_{F}\bigr) + \kappa(Y).
\]
Otherwise, we apply Proposition~\ref{caoproposition} as follows.
\begin{itemize}
    \item \textbf{Case~1}. 
\[
i\Theta_{\det h}\bigl(\det f_{*}\mathcal{O}_X\bigl(m(K_{X/Y} + \Delta)\bigr)\bigr)|_{Y_Z} \equiv 0, \hspace{0.2cm} \text{for $m$ sufficiently divisible.}
\]
By applying Proposition~\ref{caoprop2}, we conclude that
\[
f_{*}\mathcal{O}_X\bigl(m(K_{X/Y} + \Delta)\bigr)|_{Y_Z}
\]
is a Hermitian flat vector bundle, and that the induced metric \( h \) is smooth. Therefore, we apply Proposition~\ref{caoproposition}: for every
\[
s \in H^{0}\!\left(Y_Z,\, f_{*}\mathcal{O}_X\bigl(m(K_{X/Y} + \Delta)\bigr)\right)
= H^{0}\!\left(X_Z,\, m(K_X + \Delta)\right),
\]
there exists
\[
S \in H^{0}\!\left(Y,\, \omega_Y \otimes f_{*}\mathcal{O}_X\bigl(m(K_{X/Y} + \Delta)\bigr) \otimes g^{*}(cA)\right)
\]
such that
\[
S|_{Y_Z} = s \otimes g^{*}e,
\]
where \( e \in \mathcal{O}_{Z,z}(cA) \) and \( c \) is sufficiently large, depending only on \( Z \) and \( A \). Consequently, we deduce that
\begin{equation}\label{eq1}
\kappa\bigl(m(K_X + \Delta) - p h^{*}A\bigr)
\geq \kappa\bigl(K_{X_Z} + \Delta|_{X_Z}\bigr)
\end{equation}
for \( m \) sufficiently large and for some integer \( p \).

Next, by applying \cite{Kawamatacurves}, we obtain the inequality
\begin{equation}\label{eq2}
\kappa\bigl(K_{X_Z} + \Delta|_{X_Z}\bigr)
\geq \kappa\bigl(K_F + \Delta|_{F}\bigr) + \kappa(Y_Z)
= \kappa\bigl(K_F + \Delta|_{F}\bigr)
\geq 0 .
\end{equation}
Combining inequalities~\eqref{eq1} and~\eqref{eq2}, we conclude that \( m(K_X + \Delta) - p h^{*}A \) is \( \mathbb{Q} \)-effective. Hence
\[
\kappa(K_X + \Delta)
= \kappa\bigl(K_{X_Z} + \Delta|_{X_Z}\bigr) + \kappa(Y).
\]
Finally, combining this last equality with inequality~\eqref{eq2}, we obtain the desired result.

\item \textbf{Case~2}. 
\[
i\Theta_{\det h}\bigl(\det f_{*}\mathcal{O}_X(m(K_{X/Y} + \Delta))\bigr)|_{Y_Z} \text{ is ample for $m$ sufficiently divisible.}
\]

Since \( K_Y \) is semi-ample, we may choose \( \widetilde{\epsilon} > 0 \) sufficiently small such that
\[
\widetilde{\Delta} := \widetilde{\epsilon} f^{*}K_Y
\]
and \( (X, \Delta + \widetilde{\Delta}) \) is a klt pair. Then
\[
\det f_{*}\mathcal{O}_X\bigl(m(K_{X/Y} + \Delta + \widetilde{\Delta})\bigr)
= \det f_{*}\mathcal{O}_X\bigl(m(K_{X/Y} + \Delta)\bigr) \otimes \mathcal{O}_Y(m \widetilde{\epsilon} K_Y)^{\otimes r_m},
\]
where \( r_m \) is the rank of \( f_{*}\mathcal{O}_X\bigl(m(K_{X/Y} + \Delta)\bigr) \).

On the other hand, by the work of \cite{Caopaun}, we have
\[
K_{X/Y} + \Delta + \widetilde{\Delta} + E - \epsilon f^{*}\bigl(\det f_{*}\mathcal{O}_X(m(K_{X/Y} + \Delta + \widetilde{\Delta}))\bigr)
\]
is pseudo-effective for some sufficiently small \( \epsilon > 0 \), where \( E \) is a divisor on \( X \) with
\(\mathrm{codim}_Y f_{*}(E) \geq 2\). Hence,
\[
K_X + \Delta + E - \epsilon f^{*}\Bigl(\det f_{*}\mathcal{O}_X(m(K_{X/Y} + \Delta + \widetilde{\Delta})) + \frac{1 - \widetilde{\epsilon}}{\epsilon} K_Y\Bigr)
\]
is pseudo-effective. It is known that the divisor \( E \) does not contribute to the Kodaira dimension. Consequently by the bigness of 
\[
\det f_{*}\mathcal{O}_X\bigl(m(K_{X/Y} + \Delta + \widetilde{\Delta})\bigr)
\]
and by applying Theorem~\ref{schnellthm}, we obtain
\[
\kappa(K_X + \Delta) = \kappa(K_F + \Delta|_{F}) + 3.
\]
\end{itemize}
\end{proof}
As mentioned at the beginning of the section, we may assume that $K_Y$ is semi-ample. We denote by $g \colon Y \to Z$ the Iitaka fibration of $Y$, with general fiber $Y_Z$.
\begin{theorem}\label{kodairadimension1}
Let $(X,\Delta)$ be a klt pair. Let $f \colon X \to Y$ be an algebraic fiber space, where $X$ and $Y$ are smooth projective varieties of dimensions $n$ and $3$, respectively, such that $\kappa(Y)=1$, and let $F$ be a general fiber of $f$. Then
\[
\kappa(K_X + \Delta) \geq \kappa(K_F + \Delta|_{F}) + 1,
\]
provided that
\[
\kappa\!\left(\det f_{*}\mathcal{O}_X\!\left(m(K_{X/Y} + \Delta)\right)\big|_{Y_Z}\right) > 0, \hspace{0.2cm} \text{for $m$ sufficiently divisible}
\]
in the case where $Z \simeq \mathbb{P}^1$.
\end{theorem}
\begin{proof}
When $g(Z) > 0$, the desired result follows easily from Theorem~\ref{schnellthm}. The only difficult case occurs when $Z \simeq \mathbb{P}^1$. We proceed as follows.

\begin{itemize}
\item \textbf{Case 1.}
\[
\kappa\!\left(\det f_{*}\mathcal{O}_X\!\left(m(K_{X/Y} + \Delta)\right)\big|_{Y_Z}\right) = 2.
\]
In this case, the divisor
\[
\det f_{*}\mathcal{O}_X\!\left(m(K_{X/Y} + \Delta)\right)\big|_{Y_Z}
\]
is big on the minimal surface $Y_Z$. Choose $\widetilde{\varepsilon} > 0$ sufficiently small and set
\[
\widetilde{\Delta} := \widetilde{\varepsilon}\, f^{*}K_Y,
\]
so that $(X, \Delta + \widetilde{\Delta})$ is a klt pair. By the same argument as in Case~2 of Theorem~\ref{positivekodaira}, we deduce that
\[
\kappa(K_X + \Delta) = \kappa(K_F + \Delta|_{F}) + 3.
\]

\item \textbf{Case 2.}
\[
\kappa\!\left(\det f_{*}\mathcal{O}_X\!\left(m(K_{X/Y} + \Delta)\right)\big|_{Y_Z}\right) = 1.
\]
We choose $p$ sufficiently large such that the image of the map defined by
\[
\left|p\,\det f_{*}\mathcal{O}_X\!\left(m(K_{X/Y} + \Delta)\right)\big|_{Y_Z}\right|
\]
is a curve. For a sufficiently ample divisor $A$ on $Z$, we have
\[
\kappa\!\left(p\,\det f_{*}\mathcal{O}_X\!\left(m(K_{X/Y} + \Delta)\right) + g^{*}A\right) = 2.
\]
Hence, for $p$ large enough and $A$ sufficiently ample, the linear system
\[
\left|p\,\det f_{*}\mathcal{O}_X\!\left(m(K_{X/Y} + \Delta)\right) + g^{*}A\right|
\]
defines a rational map whose image is a projective surface $M$, and which factorizes the morphism $g$. The general fiber $Y_M$ is an elliptic curve (after resolving the rational map). Moreover,
\[
i\Theta_{\det h}\!\left(\det f_{*}\mathcal{O}_X\!\left(m(K_{X/Y} + \Delta)\right)\right)\big|_{Y_M} \equiv 0.
\]
By Proposition~\ref{caoprop2}, the sheaf $f_{*}\mathcal{O}_X\!\left(m(K_{X/Y} + \Delta)\right)\big|_{Y_M}$ is a Hermitian flat vector bundle with smooth induced metric $h$. Therefore, as in Case~1 of Theorem~\ref{positivekodaira}, applying Proposition~\ref{caoproposition}, we deduce that
\[
\kappa(K_X + \Delta) \geq \kappa(K_F + \Delta|_{F}) + 1.
\]
Indeed, applying Proposition~\ref{caoproposition}, we obtain
\begin{equation}\label{eq5}
\kappa(K_X + \Delta) = \kappa(K_{X_Z} + \Delta|_{X_Z}) + 1.
\end{equation}
Moreover, by \cite{Caolitaka2}, we have
\begin{equation}\label{eq6}
\kappa(K_{X_Z} + \Delta|_{X_Z}) \geq \kappa(F) + \kappa(Y_Z).
\end{equation}
Combining \eqref{eq5} and \eqref{eq6}, we conclude the desired inequality.
\end{itemize}
\end{proof}
\begin{remark}\label{remarksection5}
In the case where $\kappa(Y)=1$, $Z \simeq \mathbb{P}^1$, and
\[
\kappa\!\left(\det f_{*}\mathcal{O}_X\!\left(m(K_{X/Y} + \Delta)\right)\big|_{Y_Z}\right) \leq 0,
\]
the situation becomes more complicated, especially when $\kappa(K_F + \Delta|_{F})>0$. Nevertheless, we can still deduce some partial results.  

If
\[
\det f_{*}\mathcal{O}_X\!\left(m(K_{X/Y} + \Delta)\right)\big|_{Y_Z}
= \mathcal{O}_{Y_Z},
\]
then $f_{*}\mathcal{O}_X\!\left(m(K_{X/Y} + \Delta)\right)\big|_{Y_Z}$ is a Hermitian flat torsion-free sheaf, and by Theorem~\ref{fujitadecom}, it must be a Hermitian flat vector bundle. Applying Proposition~\ref{caoproposition}, we deduce the Iitaka inequality.

It may also happen that
\[
\kappa\!\left(\det f_{*}\mathcal{O}_X\!\left(m(K_{X/Y} + \Delta)\right)\big|_{Y_Z}\right) = -\infty,
\]
for example when $\det f_{*}\mathcal{O}_X\!\left(m(K_{X/Y} + \Delta)\right)\big|_{Y_Z}$ is a non-torsion line bundle in $\operatorname{Pic}^{0}(Y_Z)$. In this case,
\[
i\Theta_{\det h}\!\left(\det f_{*}\mathcal{O}_X\!\left(m(K_{X/Y} + \Delta)\right)\right)\big|_{Y_Z} \equiv 0.
\]
Consequently, by Proposition~\ref{caoprop2} and Theorem~\ref{fujitadecom}, we deduce that
\[
f_{*}\mathcal{O}_X\!\left(m(K_{X/Y} + \Delta)\right)\big|_{Y_Z}
\]
is a Hermitian flat vector bundle, and that the induced metric $h$ is smooth. Using the smoothness of the metric, we can apply Proposition~\ref{caoproposition} to prove the Iitaka conjecture.
\end{remark}
\section{Algebraic fiber spaces over Calabi-Yau Threefolds}
By the previous sections, one of the remaining cases in the proof of the Iitaka conjecture $C_{n,3}$ occurs when $\kappa(Y)=0$ and $q(Y)=0$. By the abundance theorem, such varieties satisfy $c_1(Y)=0$, assuming that $K_Y$ is semi-ample. Consequently, by Beauville-Bogomolov decomposition theorem, after a finite \'{e}tale cover $\widetilde{Y} \to Y$, the variety $\widetilde{Y}$ is a product of Calabi--Yau manifolds and a torus. Since the conjecture $C_{n,m}$ is preserved under smooth finite covers, by Corollary~\ref{irregthm} it remains to prove the conjecture in the case where $Y$ is a Calabi--Yau threefold.
\begin{theorem}\label{overCY}
Let $(X,\Delta)$ be a Klt pair. Let $f : X \to Y$ be an algebraic fiber space, where $X$ is a smooth projective variety of dimension $n$ and 
$Y$ is a Calabi-Yau threefold, and let $F$ be a general fiber of $f$. Then
\[
\kappa(K_X +\Delta) \geq \kappa(K_F + \Delta|_{F}),\]
provided that $\kappa(\det f_{*}\mathcal{O}_X\bigl(m(K_{X/Y} + \Delta)\bigr) > 1$ for $m$ sufficiently divisible.
\end{theorem}
\begin{proof}
It remains to prove that every section of $m(K_{F} + \Delta_{|_{F}})$ extends. We divide the proof of the desired result into the following two cases. If $$\kappa(\det f_{*}\mathcal{O}_X\bigl(m(K_{X/Y} + \Delta)\bigr)) =3.$$ Then it is known that the Conjecture holds, for instance by \cite{Caopaun}. If
\[
\kappa\!\left(\det f_{*}\mathcal{O}_X\bigl(m(K_{X/Y}+\Delta)\bigr)\right)=2.
\]
Then, for sufficiently large \(r\), the linear system \(|rN|\) defines a rational map
\[
g \colon Y \dashrightarrow Z
\]
to a rational surface \(Z\), where
\[
N := \det f_{*}\mathcal{O}_X\bigl(m(K_{X/Y}+\Delta)\bigr).
\]
Resolving the indeterminacies of \(g\), we obtain a birational morphism
\[
\phi \colon \widetilde{Y} \to Y
\]
and a morphism
\[
\widetilde{g} \colon \widetilde{Y} \to Z
\]
defined by the linear system \(|\phi^{*}(rN) - E_N|\), where \(E_N\) is an effective divisor supported on the fixed part and the \(\phi\)-exceptional locus. Therefore, we obtain the following commutative diagram:
\[
\begin{tikzcd}
\widetilde{X} \arrow[r, "\psi"] \arrow[d, "\widetilde{f}"] & X \arrow[d, "f"] \\
\widetilde{Y} \arrow[r, "\phi"] \arrow[rd, "\widetilde{g}"] & Y \arrow[d, dashed, "g"] \\
 & Z
\end{tikzcd}
\]
Here \(\widetilde{X}\) is a resolution of the fiber product \(X \times_{Y} \widetilde{Y}\), and \(\psi\) is a birational morphism.

\medskip
Now applying \cite{Caopaun}, we obtain that, for some sufficiently small \(\epsilon > 0\),
\[
K_{X/Y} + \Delta + E - \epsilon f^{*}N
\]
is pseudo-effective, where \( E \) is a divisor on \( X \) with
\(\mathrm{codim}_Y f_{*}(E) \geq 2\). Consequently,
\[
D:=\psi^{*}\bigl(K_{X/Y} + \Delta + E - \epsilon f^{*}N\bigr)
\]
is pseudo-effective. We write this pseudo-effective divisor in the following way
\[
\psi^{*}\bigl(K_{X/Y} + \Delta + E - \epsilon f^{*}N\bigr)
\]
\[
= \psi^{*}\bigl(K_{X/Y} + \Delta + E) - \frac{\epsilon}{r}\widetilde{f}^{*}E_N -\epsilon \widetilde{f}^{*}(\phi^{*}N -\frac{1}{r}E_N)
\]
Now we write
\[
K_{\widetilde{Y}} = \phi^{*}K_Y + E_Y
\]
and
\[
K_{\widetilde{X}} + \widetilde{\Delta} = \psi^{*}(K_X + \Delta),
\]
where $(\widetilde{X}, \widetilde{\Delta})$ is a klt pair. In the end, we obtain
\[
D = K_{\widetilde{X}/\widetilde{Y}} + \widetilde{\Delta} + \widetilde{f}^{*}E_Y + \psi^{*}E
- \frac{\epsilon}{r}\widetilde{f}^{*}E_N
- \epsilon\,\widetilde{f}^{*}\!\left(\phi^{*}N - \frac{1}{r}E_N\right),
\]
which is pseudo-effective.

Using an argument inspired by \cite{Caolitaka2}, adapted to our situation, we take $m_1 \gg m_2$, where $m_1$ and $m_2$ are sufficiently large, and write
\[
(m_1 + m_2)(K_{\widetilde{X}/\widetilde{Y}} + \widetilde{\Delta})
= m_1\!\left(K_{\widetilde{X}/\widetilde{Y}} + \widetilde{\Delta} + \frac{m_2}{m_1}D\right)
\]
\[
\quad + \epsilon m_2 \widetilde{f}^{*}\!\left(\phi^{*}N - \frac{1}{r}E_N\right)
- m_2 \psi^{*}E
- m_2 \widetilde{f}^{*}E_Y
+ \frac{m_2}{r}\epsilon \widetilde{f}^{*}E_N.
\]

Therefore,
\[
(m_1 + m_2)\!\left(K_{\widetilde{X}/\widetilde{Y}} + \widetilde{\Delta}
+ \frac{m_2}{m_1 + m_2}\psi^{*}E
+ \frac{m_2}{m_1 + m_2}\widetilde{f}^{*}E_Y
- \frac{m_2}{r(m_1 + m_2)}\epsilon \widetilde{f}^{*}E_N\right)
\]
\[
= m_1\!\left(K_{\widetilde{X}/\widetilde{Y}} + \widetilde{\Delta} + \frac{m_2}{m_1}D\right)
+ \epsilon m_2 \widetilde{f}^{*}\!\left(\phi^{*}N - \frac{1}{r}E_N\right).
\]

Set
\[
\bar{\Delta}
= \widetilde{\Delta}
+ \frac{m_2}{m_1 + m_2}\psi^{*}E
+ \frac{m_2}{m_1 + m_2}\widetilde{f}^{*}E_Y
- \frac{m_2}{r(m_1 + m_2)}\epsilon \widetilde{f}^{*}E_N.
\]

Since the support of $E_N$ is contained in the support of $E_Y$, and by taking $m_1$ sufficiently large, the pair $(\widetilde{X}, \bar{\Delta})$ is klt.  Therefore,
\[
\widetilde{f}_{*}\mathcal{O}_{\widetilde{X}}\!\left((m_1 + m_2)(K_{\widetilde{X}/\widetilde{Y}} + \bar{\Delta})\right)
= \widetilde{f}_{*}\mathcal{O}_{\widetilde{X}}\!\left(m_1\!\left(K_{\widetilde{X}/\widetilde{Y}} + \widetilde{\Delta} + \frac{m_2}{m_1}D\right)\right)
\]
\[
\otimes \mathcal{O}_{\widetilde{Y}}\!\left(\epsilon m_2\!\left(\phi^{*}N - \frac{1}{r}E_N\right)\right).
\]
We discuss two cases:
\begin{itemize}
    \item[a)] $c_1\!\left(\det \widetilde{f}_{*}\mathcal{O}_{\widetilde{X}}\!\left(m_1\!\left(K_{\widetilde{X}/\widetilde{Y}} + \widetilde{\Delta} + \frac{m_2}{m_1}D\right)\right)|_{\widetilde{Y}_Z}\right) \neq 0$. 
    In this case, it is clear that
    \[
    \det \widetilde{f}_{*}\mathcal{O}_{\widetilde{X}}\!\left((m_1 + m_2)(K_{\widetilde{X}/\widetilde{Y}} + \bar{\Delta})\right)
    \]
    is big. Therefore,
    \[
    \kappa(K_{\widetilde{X}} + \bar{\Delta})
    = \kappa\!\left(K_{\widetilde{F}} + \bar{\Delta}|_{{\widetilde{F}}}\right) + 3.
    \]
    However, it is not difficult to see that
    \[
    \kappa(K_{\widetilde{X}} + \bar{\Delta})
    = \kappa(K_{\widetilde{X}} + \widetilde{\Delta})
    \]
    and
    \[
    \kappa\!\left(K_{\widetilde{F}} + \bar{\Delta}_{|_{\widetilde{F}}}\right)
    \geq \kappa\!\left(K_{\widetilde{F}} + \widetilde{\Delta}|_{\widetilde{F}}\right).
    \]
    Consequently, we conclude that
    \[
    \kappa(K_{\widetilde{X}} + \widetilde{\Delta})
    \geq \kappa\!\left(K_{\widetilde{F}} + \widetilde{\Delta}|_{\widetilde{F}}\right) + 3.
    \]
    \item[b)] 
$c_1\!\left(\det \widetilde{f}_{*}\mathcal{O}_{\widetilde{X}}\!\left(m_1\!\left(K_{\widetilde{X}/\widetilde{Y}} + \widetilde{\Delta} + \frac{m_2}{m_1}D\right)\right)|_{\widetilde{Y}_Z}\right) = 0$.
By applying Proposition~\ref{caoprop2}, we deduce that
\[
\widetilde{f}_{*}\mathcal{O}_{\widetilde{X}}\!\left(m_1\!\left(K_{\widetilde{X}/\widetilde{Y}} + \widetilde{\Delta} + \frac{m_2}{m_1}D\right)\right)|_{\widetilde{Y}_Z}
\]
is Hermitian flat. Therefore, by applying \cite{Kawamatacurves} and Proposition~\ref{caoproposition} and taking
$H = (m-1)K_Y$, we obtain the desired result:
\[
\kappa(K_{\widetilde{X}} + \widetilde{\Delta})
\geq \kappa\!\left(K_{\widetilde{F}} + \widetilde{\Delta}|_{\widetilde{F}}\right).
\]
\end{itemize}
\end{proof}
\begin{proposition}\label{propoovercy}
Let $(X,\Delta)$ be a klt pair. Let $f \colon X \to Y$ be an algebraic fiber space, where $X$ is a smooth projective variety of dimension $n$ and $Y$ is a Calabi--Yau threefold. Let $F$ be a general fiber of $f$. In the setting of the proof of Theorem~\ref{overCY}, assume that
\[
\kappa\!\left(\det f_{*}\mathcal{O}_X\bigl(m(K_{X/Y} + \Delta)\bigr)\right)=1,
\]
and
\[
\kappa\!\left(
\det \widetilde{f}_{*}\mathcal{O}_{\widetilde{X}}
\!\left(
m_1\!\left(K_{\widetilde{X}/\widetilde{Y}} + \widetilde{\Delta} + \frac{m_2}{m_1}D\right)
\right)\Big|_{\widetilde{Y}_Z}
\right) \geq 1.
\]
Then
\[
\kappa(K_X + \Delta) \geq \kappa(K_F + \Delta|_{F}).
\]
\end{proposition}

\begin{proof}
Since
\[
\kappa\!\left(\det f_{*}\mathcal{O}_X\bigl(m(K_{X/Y} + \Delta)\bigr)\right)=1,
\]
the linear system
\[
\bigl|r\,\det f_{*}\mathcal{O}_X\bigl(m(K_{X/Y} + \Delta)\bigr)\bigr|
\]
defines a rational map to a curve $Z$, which is isomorphic to $\mathbb{P}^1$. As in Theorem \ref{overCY}, we obtain the following diagram:
\[
\begin{tikzcd}
\widetilde{X} \arrow[r, "\psi"] \arrow[d, "\widetilde{f}"] & X \arrow[d, "f"] \\
\widetilde{Y} \arrow[r, "\phi"] \arrow[rd, "\widetilde{g}"] & Y \arrow[d, dashed, "g"] \\
 & Z
\end{tikzcd}
\]

\begin{itemize}
\item[1)] If
\[
\kappa\!\left(
\det \widetilde{f}_{*}\mathcal{O}_{\widetilde{X}}
\!\left(
m_1\!\left(K_{\widetilde{X}/\widetilde{Y}} + \widetilde{\Delta} + \frac{m_2}{m_1}D\right)
\right)\Big|_{\widetilde{Y}_Z}
\right)=2,
\]
then
\[
\det \widetilde{f}_{*}\mathcal{O}_{\widetilde{X}}
\!\left((m_1+m_2)\bigl(K_{\widetilde{X}/\widetilde{Y}}+\overline{\Delta}\bigr)\right)
\]
is big. Hence, the result follows.

\item[2)] If
\[
\kappa\!\left(
\det \widetilde{f}_{*}\mathcal{O}_{\widetilde{X}}
\!\left(
m_1\!\left(K_{\widetilde{X}/\widetilde{Y}} + \widetilde{\Delta} + \frac{m_2}{m_1}D\right)
\right)\Big|_{\widetilde{Y}_Z}
\right)=1,
\]
then
\[
\kappa\!\left(
\det \widetilde{f}_{*}\mathcal{O}_{\widetilde{X}}
\!\left((m_1+m_2)\bigl(K_{\widetilde{X}/\widetilde{Y}}+\overline{\Delta}\bigr)\right)
\right)=2.
\]
Thus, the linear system
\[
\bigl|r\,\det \widetilde{f}_{*}\mathcal{O}_{\widetilde{X}}
\!\left((m_1+m_2)\bigl(K_{\widetilde{X}/\widetilde{Y}}+\overline{\Delta}\bigr)\right)\bigr|
\]
defines a rational map whose image is a projective surface $M$, and which factorizes the morphism $\widetilde{g}$. The general fiber $\widetilde{Y}_M$ is an elliptic curve (after resolving the rational map). Moreover,
\[
i\Theta_{\det \widetilde{h}}\!\left(
\det \widetilde{f}_{*}\mathcal{O}_{\widetilde{X}}
\!\left((m_1+m_2)\bigl(K_{\widetilde{X}/\widetilde{Y}}+\overline{\Delta}\bigr)\right)
\right)\Big|_{\widetilde{Y}_M} \equiv 0.
\]
By Proposition~\ref{caoprop2}, 
\[
\widetilde{f}_{*}\mathcal{O}_{\widetilde{X}}
\!\left((m_1+m_2)\bigl(K_{\widetilde{X}/\widetilde{Y}}+\overline{\Delta}\bigr)\right)\Big|_{\widetilde{Y}_M}
\]
is a Hermitian flat vector bundle with smooth induced metric $\widetilde{h}$. Therefore, applying Proposition~\ref{caoproposition}, we deduce
\[
\kappa(K_{\widetilde{X}} + \widetilde{\Delta})
= \kappa(K_{\widetilde{X}} + \overline{\Delta})
\geq \kappa(K_{\widetilde{F}} + \overline{\Delta}|_{\widetilde{F}}) + 1
\geq \kappa(K_{\widetilde{F}} + \widetilde{\Delta}|_{\widetilde{F}}) + 1.
\]
\end{itemize}
\end{proof}
\begin{remark}\label{remark1overcy}
At this point, we are not able to omit the condition
\[
\kappa\!\left(
\det \widetilde{f}_{*}\mathcal{O}_{\widetilde{X}}
\!\left(
m_1\!\left(K_{\widetilde{X}/\widetilde{Y}} + \widetilde{\Delta} + \frac{m_2}{m_1}D\right)
\right)\Big|_{\widetilde{Y}_Z}
\right) \geq 1
\]
in Proposition~\ref{propoovercy}. We can only provide some evidence in the case where
\[
\kappa\!\left(
\det \widetilde{f}_{*}\mathcal{O}_{\widetilde{X}}
\!\left(
m_1\!\left(K_{\widetilde{X}/\widetilde{Y}} + \widetilde{\Delta} + \frac{m_2}{m_1}D\right)
\right)\Big|_{\widetilde{Y}_Z}
\right) < 1,
\]
for example, when
\[
\det \widetilde{f}_{*}\mathcal{O}_{\widetilde{X}}
\!\left(
m_1\!\left(K_{\widetilde{X}/\widetilde{Y}} + \widetilde{\Delta} + \frac{m_2}{m_1}D\right)
\right)\Big|_{\widetilde{Y}_Z}
\]
is trivial or a non-torsion element in $\operatorname{Pic}^{0}(\widetilde{Y}_Z)$. In these cases, the Iitaka conjecture holds.
\end{remark}
\begin{remark}\label{remarkovercy2}
 If $\kappa\bigl(\det f_{*}\mathcal{O}_X\bigl(m(K_{X/Y} + \Delta)\bigr)\bigr) = 0$, then 
\[
\det f_{*}\mathcal{O}_X\bigl(m(K_{X/Y} + \Delta)\bigr) \equiv 0
\]
outside finitely many divisors. Hence,
\[
f_{*}\mathcal{O}_X\bigl(m(K_{X/Y} + \Delta)\bigr)
\]
is a Hermitian flat torsion free sheaf outside finitely many divisors. Denote
\[
\mathcal{C}:=\{T >0, \hspace{0.2cm} T\in H^{0}(Y, \det f_{*}\mathcal{O}_X\bigl(m(K_{X/Y} + \Delta)\bigr)\bigr)\}
\]
If $\mathcal{C} = \emptyset$, then applying Theorem~\ref{fujitadecom}, we deduce that
\[
f_{*}\mathcal{O}_X\bigl(m(K_{X/Y} + \Delta)\bigr)
\]
is a Hermitian flat vector bundle, which is in fact trivial since $Y$ is a Calabi--Yau threefold. Consequently 
\[
\kappa(X) = \kappa(F).
\]
\end{remark}
\subsection*{Conclusion}
In this section, we have completed the proof of the Iitaka conjecture over Calabi--Yau threefolds in the case where
\[
\kappa\!\left(\det f_{*}\mathcal{O}_X\bigl(m(K_{X/Y} + \Delta)\bigr)\right) > 1,
\]
see Theorem~\ref{overCY}. We have also obtained partial results in the cases
\[
\kappa\!\left(\det f_{*}\mathcal{O}_X\bigl(m(K_{X/Y} + \Delta)\bigr)\right) = 0,1,
\]
see Proposition~\ref{propoovercy} and Remarks~\ref{remark1overcy} and~\ref{remarkovercy2}.

For such fibrations, we conjecture that
\[
\det f_{*}\mathcal{O}_X\bigl(m(K_{X/Y}+\Delta)\bigr)
\]
is $\mathbb{Q}$-effective. Even if we assume that
\[
\det f_{*}\mathcal{O}_X\bigl(m(K_{X/Y}+\Delta)\bigr)
\]
is nef, it is not clear that it is $\mathbb{Q}$-effective, since this would follow from the abundance conjecture for Calabi--Yau threefolds. However, if we assume that
\[
\det f_{*}\mathcal{O}_X\bigl(m(K_{X/Y}+\Delta)\bigr)
\]
is nef and satisfies
\[
\det f_{*}\mathcal{O}_X\bigl(m(K_{X/Y}+\Delta)\bigr)\cdot c_2(Y) > 0,
\]
or in particular for the case that the Calabi--Yau threefold $Y$ satisfies $c_2(Y) > 0$, then by applying the Riemann--Roch theorem for threefolds we can deduce that
\[
\det f_{*}\mathcal{O}_X\bigl(m(K_{X/Y}+\Delta)\bigr)
\]
is effective.

Indeed, the case
\[
\nu\!\left(\det f_{*}\mathcal{O}_X\bigl(m(K_{X/Y}+\Delta)\bigr)\right)=2
\]
follows from  Kawamata--Viehweg vanishing theorem, in this situation, the divisor
\[
\det f_{*}\mathcal{O}_X\bigl(m(K_{X/Y}+\Delta)\bigr)
\]
defines an elliptic fibration whose base is a rational surface with only quotient singularities.

The case
\[
\nu\!\left(\det f_{*}\mathcal{O}_X\bigl(m(K_{X/Y}+\Delta)\bigr)\right)=1
\]
requires additional work and is treated in~\cite{oguisok},  the divisor
\[
\det f_{*}\mathcal{O}_X\bigl(m(K_{X/Y}+\Delta)\bigr)
\]
induces a fibration of the Calabi--Yau threefold $Y$ whose general fiber is a K3 surface.

Note that when
\[
\nu\!\left(\det f_{*}\mathcal{O}_X\bigl(m(K_{X/Y}+\Delta)\bigr)\right)=0,
\]
the divisor
\[
\det f_{*}\mathcal{O}_X\bigl(m(K_{X/Y}+\Delta)\bigr)
\]
is trivial, while in the case
\[
\nu\!\left(\det f_{*}\mathcal{O}_X\bigl(m(K_{X/Y}+\Delta)\bigr)\right)=3,
\]
it is big.
\begin{remark}
The condition
\[
\det f_{*}\mathcal{O}_X\bigl(m(K_{X/Y}+\Delta)\bigr)\cdot c_2(Y) > 0
\]
is quite strong. If this condition is not assumed, there are results proving the effectivity of
\[
\det f_{*}\mathcal{O}_X\bigl(m(K_{X/Y}+\Delta)\bigr)
\]
in special situations where
\[
\nu\!\left(\det f_{*}\mathcal{O}_X\bigl(m(K_{X/Y}+\Delta)\bigr)\right)=2.
\]
Indeed, if we assume that $c_3(Y)\neq 0$ and that there exists an ample divisor $H$ on $Y$ such that the restriction
\[
\det f_{*}\mathcal{O}_X\bigl(m(K_{X/Y}+\Delta)\bigr)\big|_G
\]
is ample for a general member $G\in |H|$, then
\[
\det f_{*}\mathcal{O}_X\bigl(m(K_{X/Y}+\Delta)\bigr)
\]
is semi-ample. This result is proved in~\cite{svaldiliu}.
\end{remark}
\bibliographystyle{plain}
\bibliography{slope_paper}
\end{document}